\documentclass[12pt]{article}
\usepackage{amsmath}
\usepackage{amssymb}
\usepackage{amsmath,amssymb,amsbsy,amsfonts,amsthm,latexsym,
amsopn,amstext,amsxtra,euscript,amscd}

\begin{document}

\def\A{\mathbb{A}}
\def\B{\mathbf{B}}
\def \C{\mathbb{C}}
\def \F{\mathbb{F}}
\def \K{\mathbb{K}}

\def \Z{\mathbb{Z}}
\def \P{\mathbb{P}}
\def \R{\mathbb{R}}
\def \Q{\mathbb{Q}}
\def \N{\mathbb{N}}
\def \Z{\mathbb{Z}}

\def\B{\mathcal B}
\def\e{\varepsilon}

\def\cA{{\mathcal A}}
\def\cB{{\mathcal B}}
\def\cC{{\mathcal C}}
\def\cD{{\mathcal D}}
\def\cE{{\mathcal E}}
\def\cF{{\mathcal F}}
\def\cG{{\mathcal G}}
\def\cH{{\mathcal H}}
\def\cI{{\mathcal I}}
\def\cJ{{\mathcal J}}
\def\cK{{\mathcal K}}
\def\cL{{\mathcal L}}
\def\cM{{\mathcal M}}
\def\cN{{\mathcal N}}
\def\cO{{\mathcal O}}
\def\cP{{\mathcal P}}
\def\cQ{{\mathcal Q}}
\def\cR{{\mathcal R}}
\def\cS{{\mathcal S}}
\def\cT{{\mathcal T}}
\def\cU{{\mathcal U}}
\def\cV{{\mathcal V}}
\def\cW{{\mathcal W}}
\def\cX{{\mathcal X}}
\def\cY{{\mathcal Y}}
\def\cZ{{\mathcal Z}}

\def\f{\frac{|\A||B|}{|G|}}
\def\AB{|\A\cap B|}
\def \Fq{\F_q}
\def \Fqn{\F_{q^n}}

\def\({\left(}
\def\){\right)}
\def\fl#1{\left\lfloor#1\right\rfloor}
\def\rf#1{\left\lceil#1\right\rceil}
\def\Res{{\mathrm{Res}}}

\newcommand{\comm}[1]{\marginpar{
\vskip-\baselineskip \raggedright\footnotesize
\itshape\hrule\smallskip#1\par\smallskip\hrule}}

\newtheorem{lem}{Lemma}
\newtheorem{lemma}[lem]{Lemma}
\newtheorem{prop}{Proposition}
\newtheorem{proposition}[prop]{Proposition }
\newtheorem{thm}{Theorem}
\newtheorem{theorem}[thm]{Theorem}
\newtheorem{cor}{Corollary}
\newtheorem{corollary}[cor]{Corollary}
\newtheorem{prob}{Problem}
\newtheorem{problem}[prob]{Problem}
\newtheorem{ques}{Question}
\newtheorem{question}[ques]{Question}
\newtheorem{rem}{Remark}

\title{Congruences involving product of intervals and sets with small multiplicative doubling modulo a prime and applications}

\author{
{\sc J. Cilleruelo} and {\sc M.~Z.~Garaev} }

\date{}

\maketitle

\begin{abstract}
In the present paper we obtain new upper bound estimates for the
number of solutions of the congruence
$$
x\equiv y r\pmod p;\quad x,y\in \mathbb{N},\quad x,y\le H,\quad
r\in\cU,
$$
for certain ranges of $H$ and $|\cU|$, where $\cU$ is a subset of
the field of residue classes modulo $p$ having small multiplicative
doubling. We then use this estimate to show that the number of
solutions of the congruence
$$
x^n\equiv \lambda\pmod p; \quad  x\in \N,  \quad L<x<L+p/n,
$$
is at most $p^{\frac{1}{3}-c}$ uniformly over positive integers $n,
\lambda$ and $L$, for some absolute constant $c>0$. This implies, in
particular, that if $f(x)\in \Z[x]$ is a fixed polynomial without
multiple roots in $\C$, then the congruence $ x^{f(x)}\equiv 1\pmod p,
\,x\in \mathbb{N}, \,x\le p,$ has at most $p^{\frac{1}{3}-c}$
solutions as $p\to\infty$, improving some recent results of
Kurlberg, Luca and Shparlinski and of Balog, Broughan and
Shparlinski. We use our results to show that almost all the residue
classes modulo $p$ can be represented in the form $xg^y \pmod p$
with positive integers $x<p^{5/8+\varepsilon}$ and $y<p^{3/8}$. Here
$g$ denotes  a primitive root modulo $p$. We also prove that almost
all the residue classes modulo $p$ can be represented in the form
$xyzg^t \pmod p$ with positive integers
$x,y,z,t<p^{1/4+\varepsilon}$.
\end{abstract}

\section{Introduction}

In what follows, $\varepsilon$ is a small fixed positive quantity,
$\F_p$ is the field of residue classes modulo a prime number $p$,
which we consider to be sufficiently large  in terms of
$\varepsilon$. The notation $A\lesssim B$ is used to denote that
$|A|<|B|p^{o(1)}$, or equivalently,  for any $\varepsilon>0$ there
is a constant $c=c(\varepsilon)$ such that
$|A|<c|B|p^{\varepsilon}$. Given sets $\cA$ and $\cB$ their
product-set $\cA\cdot\cB$ is defined by
$$
\cA\cdot\cB=\{ab;\quad a\in \cA,\, b\in \cB\}.
$$

The distributional properties of powers of a primitive root modulo
$p$ and subgroups of $\F_p^*$ has a long story, starting from the
work of Vinogradov~\cite{Vin} in 1926. A substantial amount of
information and results can be found in the book of Konyagin and
Shparlinski~\cite{KSh}. In the present paper we continue the
investigation on this topic. Our first result is closely related to
the work of Bourgain, Konyagin and Shparlisnki~\cite{BKSh} and to
some results of Konyagin and Shparlinski from~\cite{KSh}.

\begin{theorem}
\label{thm:Main0} Let $H$ be a positive integer  and let $\cU\subset
\F_p^*$ be such that
$$
|\cU\cdot\cU|< 10|\cU|.
$$
Denote by  $J$ the number of solutions of the congruence
\begin{equation}
\label{eqn:x=yg^z}
x\equiv y r\pmod p;\quad x,y\in \mathbb{N},\quad x,y\le H,\quad r\in\cU.
\end{equation}
Then the following two assertions hold:
\begin{itemize}
\item[(i).] If for some positive integer constant $n$ we have
$$
|\cU|<p^{n/(2n+1)},\quad |\cU|H^n<p,
$$
 then $J\lesssim H$.
 \item[(ii).] If $|\cU|<p^{2/5}$, then
$$
J\lesssim H+\frac{|\cU|H^2}{p}+\frac{|\cU|^{3/4}H}{p^{1/4}}.
$$
\end{itemize}
\end{theorem}

\begin{corollary}
\label{cor:Main1} Let $H$ be a positive integer  and let $\cU\subset
\F_p^*$ be such that
$$
|\cU\cdot\cU|< 10|\cU|.
$$
Denote by  $J$ the number of solutions of the congruence
\begin{equation}
\label{eqn:xg^z=xg^z}
xr\equiv x_1r_1\pmod p;\quad x,x_1\in\N, \quad x,x_1\le H,\quad r,r_1\in\cU.
\end{equation}
Then the following two assertions hold:
\begin{itemize}
\item[(i).] If for some positive integer constant $n$ we have
$$
|\cU|<p^{n/(2n+1)},\qquad |\cU|H^n<p,
$$
 then $J\lesssim H|\cU|$.
 \item[(ii).] If $$
p^{1/3}<|\cU|<p^{2/5},\qquad |\cU|H<p,
$$
then
$$
J\lesssim \frac{|\cU|^{7/4}H}{p^{1/4}}.
$$
\end{itemize}
\end{corollary}

We remark that in the case $\cU$ is a subgroup and $n=1$ the
statement of the part $(i)$ of our Theorem~\ref{thm:Main0}
follows from Corollary~7.7 of the aforementioned
book~\cite{KSh}.

The proof of Theorem~\ref{thm:Main0} is based on ideas and
results of Bourgain, Konyagin and Shparlinski~\cite{BKSh}.
Nevertheless, in the indicated range of parameters, the upper
bound estimate of our Theorem~~\ref{thm:Main0} improves one of
the main results of~\cite{BKSh}.

We give several new applications of Theorem~\ref{thm:Main0}. Let
$d\in\N$ and $\lambda$ be an integer coprime to $p$. For real
numbers $L$ and $N\ge 1$, consider the problem of upper bound
estimates for the number $T_p(d,\lambda, L, N)$ of solutions of
the congruence
\begin{equation}
\label{eqn:x^d=lambda}
x^d\equiv \lambda\pmod p; \quad x\in\N, \quad L+1\le x\le L+N.
\end{equation}
Trivially, for $N<p$, we have the  bound $T_p(d, \lambda, L,
N)\le \min\{d, N\}$. The problem of obtaining nontrivial upper
bounds for $T(d, \lambda, L, N)$ is of a very high interest,
with a variety of results in the literature, see, for example,
the aforementioned work~\cite{BKSh}, and more recent work of
Shkredov~\cite{Shk}. Several nontrivial results can also be
derived using the arguments from~\cite{BGKS2}. For instance, it
is possible to prove that if $N<p^{2/5}$, then one has the bound
$T_p(d, \lambda, L, N)\lesssim d^{1/2}.$ Using our
Theorem~\ref{thm:Main0} we shall obtain the following new result
on $T_p(d, \lambda, L, N)$ for any range $d,N$ with $dN<p$.

\begin{theorem}
\label{thm:x^n=lambda}  There exists an absolute constant $c>0$
such that
$$
T_p(d, \lambda, L, p/d) < p^{\frac{1}{3}-c}.
$$
\end{theorem}

From Theorem~\ref{thm:x^n=lambda}  we can derive the following
consequence.
\begin{corollary}
\label{cor:x f(x)} Let $f(x)\in\Z[x]$ be a fixed non-constant
polynomial without multiple roots in $\C$. Then the congruence
\begin{equation}
\label{eqn:Th2}
x^{f(x)}\equiv 1\pmod p;\quad x\in \mathbb{N},\quad x\le p,
\end{equation}
has at most $p^{\frac{1}{3}-c}$ solutions as $p\to\infty$, for
some absolute constant $c>0$.
\end{corollary}
Corollary~\ref{cor:x f(x)} improves the upper bound of the size
$p^{6/13+o(1)}$ obtained by Kurlberg, Luca and
Shparlinski~\cite{KLSh}. We remark that the upper bound of the
size $p^{1/3+o(1)}$ was known in the particular case $f(x)=x$
from the work of Balog, Broughan and Shparlinski~\cite{BBSh}.
Our result improves this too.

The constant $c$ in Theorem~\ref{thm:x^n=lambda} and
Corollary~\ref{cor:x f(x)} can easily be made explicit. In the
special case $f(x)=x$, using a different approach, in
Corollary~\ref{cor:x f(x)} we can obtain the upper bound of the
size $p^{27/82+o(1)}$. We hope to deal with these questions
elsewhere.

We shall give two more applications of Theorem~\ref{thm:Main0}.
Let $$ \cI=\{1,2,3,\ldots,H\}\pmod p$$ be an interval of $\F_p$
with $|\cI|=H$ elements. Denote by $\cG$ either a subgroup of
$\F_p^*$ or the set
$$
\{1,g, g^2,\ldots, g^{N-1}\}\pmod p
$$
with $|\cG|=N$ elements, formed with powers of a primitive root
$g$ modulo $p$.

\begin{theorem}
\label{thm:Main1} For any fixed $\varepsilon>0$, if
$|\cI|>p^{5/8+\varepsilon},\, |\cG|>p^{3/8}$, then
$$
|\cI \cdot \cG|=p+O(p^{1-\delta})
$$
for some  $\delta=\delta(\varepsilon)>0.$
\end{theorem}

\begin{theorem}
\label{thm:Main2} For any fixed $\varepsilon>0$, if
$|\cI|>p^{1/4+\varepsilon},\, |\cG|>p^{1/4}$, then
$$
|\cI \cdot \cI\cdot \cI\cdot \cG|=p+O(p^{1-\delta})
$$
for some  $\delta=\delta(\varepsilon)>0.$
\end{theorem}

Let us mention several results relevant to
Theorems~\ref{thm:Main1},~\ref{thm:Main2}. In~\cite{GK1} it was
shown that if $|\cI|>p^{\frac{2}{3}-\frac{1}{192}+\varepsilon}$
and $\cA$ is an arbitrary subset of $\F_p$ with
$|\cA|>p^{\frac{2}{3}-\frac{1}{192}+\varepsilon}$ then
$$
|\cI\cdot \cA|=p+O(p^{1-\delta});\quad \delta=\delta(\varepsilon)>0.
$$
Later, the exponent $\frac{2}{3}-\frac{1}{192}$ was improved by
Bourgain to $\frac{5}{8}$ (unpublished). We also mention that if
$|\cI|>p^{1/2+\varepsilon}$, then one has
$$
|\cI\cdot \cI|=p+O(p^{1-\delta});\quad \delta=\delta(\varepsilon)>0,
$$
see, for example,~\cite{GK2}.

Theorem~\ref{thm:Main1} and its proof imply that if
$|\cI|>p^{5/8+\varepsilon},\, |\cG|>p^{3/8}$, then
$$
\F_p^*\subset \cI \cdot \cI \cdot \cG.
$$
We remark that from the arguments of Heath-Brown~\cite{HB} it
follows that if $|\cI|>p^{5/8+\varepsilon}$, then one has
$$
\F_p^*\subset \cI \cdot \cI \cdot \cI.
$$
Theorem~\ref{thm:Main2} can be compared with the result
from~\cite{Gar2}, where it was shown that under the same
condition $|\cI|>p^{1/4+\varepsilon}$ one has
$$
|\cI \cdot \cI\cdot \cI\cdot \cI|=p+O(p^{1-\delta});\quad \delta=\delta(\varepsilon)>0.
$$
However, the presence of $\cG$ in our theorems is an additional
obstacle which we are able to overcome using
Theorem~\ref{thm:Main0}.

Theorem~\ref{thm:Main2} implies, in particular, that any
$\lambda\not\equiv 0\pmod p$ can be represented in the form
$$
\lambda\equiv xyzuvwg^t\pmod p
$$
for some positive integers $x,y,z,t,u,v,w<p^{1/4+\varepsilon}.$

In passing, we remark that in Theorem~\ref{thm:Main0} the
condition $|\cU\cdot\cU|<10|\cU|$  can be relaxed up to
$|\cU\cdot\cU|<|\cU|p^{o(1)}$. However, the formulation in the
form $|\cU\cdot\cU|<10|\cU|$ already applies for the set $\cG$
and is sufficient to prove
Theorems~\ref{thm:x^n=lambda},~\ref{thm:Main1},~\ref{thm:Main2}.

In what follows $\chi$ denotes a character modulo the prime
number $p$ and $\chi_0$ denotes the principal character.

\section{Auxiliary Lemmas}

We start with the following lemma of Bourgain, Konyagin and
Shparlinski~\cite{BKSh}, see also~\cite{JC} for a different
proof with refined constants.

\begin{lemma}
\label{lem:BKSh} Let $\cA$ be a non-empty subset of $\cF_Q$,
where
$$
\cF_Q=\Bigl\{\frac{r}{s};\, r,s\in \mathbb{N},\, \gcd(r,s)=1,\, r,s\le Q \Bigr\}
$$
is the set of Farey fractions of order $Q$. Then for a given
positive integer $m$, the $m$-fold product set $\cA^{(m)}$ of
$\cA$ satisfies
$$
|\cA^{(m)}|>\exp\Bigl(-C(m)\frac{\log Q}{\sqrt{\log\log Q}}\Bigr)|\cA|^m,
$$
where $C(m)>0$ depends only on $m$, provided that $Q$ is large
enough.
\end{lemma}

We recall that  the $m$-fold product set $\cA^{(m)}$ of $\cA$ is
defined as
$$
\cA^{(m)}=\{a_1\cdots a_m;\quad a_1,\ldots, a_m\in \cA\}.
$$

Our next lemma stems from the work of Bourgain et.
al.~\cite{BGKS2}. Recall that a lattice in $\R^n$ is an additive
subgroup of $\R^n$ generated by $n$ linearly independent
vectors. Take an arbitrary convex compact and symmetric with
respect to $0$ body $D\subset\R^n$. Recall that, for a lattice
$\Gamma\subset\R^n$ and $i=1,\ldots,n$, the $i$-th successive
minimum $\lambda_i(D,\Gamma)$ of the set $D$ with respect to the
lattice $\Gamma$ is defined as the minimal number $\lambda$ such
that the set $\lambda D$ contains $i$ linearly independent
vectors of the lattice $\Gamma$. Obviously,
$\lambda_1(D,\Gamma)\le\ldots\le\lambda_n(D,\Gamma)$. From
~\cite[Proposition~2.1]{BHW} it is known that
\begin{equation}
\label{eqn:minimaslattice}
\prod_{i=1}^n \min\{\lambda_i(D,\Gamma),1\} \le \frac{(2n+1)!!}{
|D\cap\Gamma|}.
\end{equation}

\begin{lemma}
\label{lem:x=sy geometry} For any $s_0\in \F_p$, the number of
solutions of the congruence
\begin{equation}
\label{eqn:x=sy}
x\equiv s_0y\pmod p;\quad x,y\in \mathbb{N},\quad x\le X,\quad y\le Y,\quad \gcd(x,y)=1,
\end{equation}
is bounded by $O(1+XYp^{-1}).$
\end{lemma}

\begin{proof} Consider the lattice
$$
\Gamma=\{(u,v)\in \Z^2; \quad u\equiv s_0v\pmod p\}
$$
and the body
$$
D=\{(u,v)\in \R^2;\quad |u|\le X,\,\, |v|\le Y\}.
$$
Let $\lambda_1,\lambda_2$ be the consecutive minimas of the body
$D$ with respect to the lattice $\Gamma$. If $\lambda_2>1$ then
there is at most one independent vector in $\Gamma\cap D$,
implying that $J\le 1$, where $J$ is the number of solutions
of~\eqref{eqn:x=sy}.

Let now $\lambda_2\le 1.$ Then, by~\eqref{eqn:minimaslattice},
we get
$$
\lambda_1\lambda_2\le \frac{15}{|\Gamma\cap D|}\le \frac{15}{J}.
$$
Let $(u_i,v_i)\in \lambda_i D\cap\Gamma,\, i=1,2,$ be linearly
independent. Then
$$
0\not=\det \(\begin{array}{cc}
    u_1&v_1\\
    u_2&v_2\\
  \end{array}\)\equiv 0\pmod p.
$$
Therefore,
$$
p\le \Bigl|\det \(\begin{array}{cc}
    u_1&v_1\\
    u_2&v_2\\
  \end{array}\)\Bigr|=|u_1v_2-u_2v_1|\le 2\lambda_1\lambda_2 XY\le \frac{30XY}{J}
$$
and the result follows.

\end{proof}

We also need the following simple lemma.

\begin{lemma}
\label{lem:simple u/v=a} Let $X,Y$ be positive numbers with
$XY<p$. Then for any $\lambda$ there is at most one solution to
the congruence
$$
\frac{x}{y}\equiv\lambda\pmod p;\quad x,y\in \mathbb{N},\quad x\le X,\ y\le Y,\quad \gcd(x,y)=1.
$$
\end{lemma}
\begin{proof}
Assuming that there is at least one solution $(x_0,y_0)$, we get
that
$$
xy_0\equiv x_0y\pmod p,
$$
and since the both hand sides are not greater than $XY< p$, the
congruence is converted to an equality, which together with
$\gcd(x,y)=\gcd(x_0,y_0)=1$ implies that $x=x_0,\, y=y_0$.
\end{proof}

To prove our Theorem~\ref{thm:x^n=lambda}, we shall need the
following lemma, which can be found in~\cite[Chapter~1,
Theorem~1]{Mont}.
\begin{lem}
\label{lem:Mont} Let $\gamma_1, \ldots, \gamma_d$ be a sequence
of $d$ points of the unit interval $[0,1]$. Then for any integer
$K\ge 1$, and an interval $[\alpha, \beta] \subseteq [0,1]$, we
have
\begin{equation*}
\begin{split}
\# \{n =1, \ldots, d~:&~\gamma_n  \in [\alpha, \beta]\} - d(\beta - \alpha)\\
\ll \frac{d}{K} + &\sum_{k=1}^K \(\frac{1}{K} +\min\{\beta - \alpha, 1/k\}\)
\left|\sum_{n=1}^d \exp(2 \pi i k \gamma_n)\right|.
\end{split}
\end{equation*}
\end{lem}

We shall also need the well-known character sum bounds of
Burgess~\cite{Bur1, Bur2}.
\begin{lem}
\label{lem:Burgess} For any fixed positive integer constant $r$
the following bound holds:
$$
\max_{\chi\not=\chi_0}\Bigl|\sum_{n=L+1}^{L+N}\chi(n)\Bigr|<N^{1-\frac{1}{r}}p^{\frac{r+1}{4r^2}+o(1)}.
$$
\end{lem}

We also recall the following bound of exponential sum estimates
over subgroups due to Konyagin~\cite{Kon1}.

\begin{lem}
\label{lem:Kon1} If $\cG$ is a subgroup of $\F_p^*$ with
$|\cG|<p^{1/2}$ then
$$
\max_{a\not\equiv 0(\bmod p)}\Bigl|\sum_{x\in G}e_p{(ax)}|\ll |G|^{29/36}p^{1/18}.
$$
\end{lem}
Here and below we use the abbreviation $e_p(z)=e^{2\pi i z/p}$.
Lemma~\ref{lem:Kon1} will be used in the proof of
Theorem~\ref{thm:x^n=lambda}. We recall that a better bound
follows from the work of Shteinikov~\cite{Sht}, but since in
Theorem~\ref{thm:x^n=lambda} we do not specify the constant $c$,
for our current purposes Lemma~\ref{lem:Kon1} suffices.

\section{Proof of Theorem~\ref{thm:Main0}}

Given a positive integer $d$ we let $J_d(H,\cU)$ be the number
of solutions of \eqref{eqn:x=yg^z} with the additional condition
$\gcd(x,y)=d$. Then we have
\begin{equation}
\label{eqn:casegcd=1}
J=\sum_{d\le H}J_d(H,\cU)=\sum_{d\le H}J_1(H/d,\cU).
\end{equation}
Since each pair of relatively prime positive integers $(x,y)$
can be defined by the rational number $x/y$, it follows that
$J_1(H/d,\cU)$ is equal to the cardinality of the set
$$
\mathcal J_d=\Bigl\{\frac{x}{y};\quad x,y\in \mathbb{N},\, x,y\le \frac{H}{d},\, \gcd(x,y)=1,\; \frac{x}{y}\!\!\!\pmod
p\in \cU\Bigr\}.
$$
We observe that the $m$-fold product set $\cJ_d^{(m)}$ satisfies
$$
\mathcal J_d^{(m)}\subset \Bigl\{\frac{u}{v};\quad u,v\in \mathbb{N},\,
u,v\le (H/d)^m,\, \gcd(u,v)=1,\, \frac{u}{v}\!\!\!\pmod p\in \cU^{(m)}\Bigr\}.
$$
The  Pl\"unecke inequality (see,~\cite[Theorem 7.7]{Nat})
together with the condition $|\cU\cdot\cU|< 10 |\cU|$ implies
that $|\cU^{(m)}|< 10^m|\cU|$. Thus, using Lemma~\ref{lem:x=sy
geometry}, we derive that
$$
|\mathcal J_d^{(m)}|\ll \sum_{r\in \cU^{(m)}}\left (1+\frac{(H/d)^{2m}}{p}\right )\ll |\cU|\left (1+\frac{(H/d)^{2m}}{p}\right ),
$$
the implied constant may depend on $m.$ On the other hand
Lemma~\ref{lem:BKSh} implies that $|\mathcal J_d^{(m)}|\gtrsim
|\cJ_d|^{m}$. Thus,
\begin{equation}
\label{eqn:th1casegenJ1(H/d)}
J_1(H/d,\cU)=|\cJ_d|\lesssim |\cU|^{1/m}+\frac{|\cU|^{1/m}H^{2}}{p^{1/m}d^2}.
\end{equation}
We first prove the part $(i)$ of our theorem. It suffices to
prove that for any $\delta>0$ there exists $c=c(\delta)>0$ such
that $J_1(H/d,\cU)<(H/d)p^{\delta}$. In particular, since
$J_1(H/d,\cU)\le (H/d)^2$, we can assume that $H/d>p^{\delta}$.
Let $m$ be the smallest positive integer such that $
|\cU|<(H/d)^m.$ Clearly, $m<1+1/\delta$. It is easy to see that
$(H/d)^{m}<p/|\cU|$. Indeed, if $(H/d)^{m}\ge p/|\cU|$, then
from the condition of the theorem it follows that $m\ge n+1$. On
the other hand, by the definition of  $m$ we have
$(H/d)^{m-1}\le |\cU|$. Hence, the lower and the upper bounds
for $H/d$ give
$$
|\cU|\ge p^{(m-1)/(2m-1)}\ge p^{n/(2n+1)},
$$
which contradicts the condition of the theorem.

Thus, we have
$$
|\cU|<(H/d)^m<p/|\cU|.
$$
Combining this with~\eqref{eqn:th1casegenJ1(H/d)}, we get that
$$
J_1(H/d,\cU)\lesssim \frac{H}{d},
$$
which, in view of the remark above, finishes the proof of the
part $(i)$ of the theorem.

Now we prove the part $(ii)$ of the theorem.  In the
inequality~\eqref{eqn:casegcd=1} we split the summation over
$d<H$ into at most $H^{o(1)}$ dyadic intervals of the form
$[H/2^{j}, H/2^{j-1}]$. It then follows
from~\eqref{eqn:casegcd=1} that for some $1\le L\le H$ one has
$$
J\lesssim \sum_{H/(2L)\le d\le H/L}J_1(H/d,\cU).
$$
Using~\eqref{eqn:th1casegenJ1(H/d)} we get that for any fixed
positive integer constant $m$ we have the bound
\begin{equation}
\label{eqn:J lessim with m}
J\lesssim H\Bigl(\frac{|\cU|^{1/m}}{L}+\frac{|\cU|^{1/m}L}{p^{1/m}}\Bigr).
\end{equation}
We can assume that $L>p^{\delta}$ for some small positive
constant $\delta>0$, as otherwise, the result follows
from~\eqref{eqn:J lessim with m} for a sufficiently large
constant $m$.

If $L\ge |\cU|$, then applying~\eqref{eqn:J lessim with m} with
$m=1$ and using $L\le H$ we obtain that
$$
J\lesssim H+\frac{|\cU|H^2}{p}.
$$
Thus in this case we get the desired estimate. So, in what
follows, we assume that $L\le |\cU|$. Consider two cases.

{\bf Case 1.} $|\cU|^{1/2}\le L\le |\cU|.$ Since we also have
$L\le H$, taking $m=1$ we get
$$
J\lesssim \frac{H|\cU|}{L}+\frac{|\cU|H^2}{p}.
$$
Now take $m=2$ in~\eqref{eqn:J lessim with m} and get
$$
J\lesssim  H+\frac{|\cU|^{1/2}HL}{p^{1/2}}.
$$
Putting the last two inequalities together, we obtain that
$$
J\lesssim  H+\frac{|\cU|H^2}{p}+\min\Bigl\{\frac{H|\cU|}{L}, \frac{|\cU|^{1/2}HL}{p^{1/2}}\Bigr\}.
$$
Since
$$
\min\Bigl\{\frac{H|\cU|}{L}, \frac{|\cU|^{1/2}HL}{p^{1/2}}\Bigr\}\le\frac{H|\cU|^{3/4}}{p^{1/4}},
$$
the result follows in this case.

{\bf Case 2.} $|\cU|^{1/(n+1)}\le L\le |\cU|^{1/n}$ for some
integer $n\ge 2$. Since $L>p^{\delta}$ for some positive
constant $\delta$, we get that $n\le n_0$ for some  integer
constant $n_0$. We apply the bound~\eqref{eqn:J lessim with m}
with $m=n+1$ and obtain
$$
J\lesssim H+\frac{H|\cU|^{(2n+1)/n(n+1)}}{p^{1/(n+1)}}.
$$
Since $n\ge 2$ and $|\cU|<p^{2/5}$ we get
$$
|\cU|^{(2n+1)/n}\le p.
$$
Therefore, we obtain $J\lesssim H$ and finish the proof of our
theorem.

To derive Corollary~\ref{cor:Main1}, we fix $r=r_0\in \cU$ such
that
$$
J\le |\cU|J',
$$
where $J'$ is the number of solutions of the congruence
$$
x\equiv x_1r_0^{-1}r_1\pmod p;\quad 1\le x,x_1\le H,\quad r_1\in\cU.
$$
Now we simply denote
$$
\cU'=\{r_0^{-1}r_1;\quad r_1\in \cU\}
$$
and apply Theorem~\ref{thm:Main0} with $\cU$ substituted by
$\cU'$.

\section{Proof of Theorem~\ref{thm:x^n=lambda} and Corollary~\ref{cor:x f(x)}}

We can assume that $\lambda\not\equiv 0\pmod p$. Denote
$N=\lfloor p/d\rfloor$. If $d_1=(d,p-1)$, then the
congruence~\eqref{eqn:x^d=lambda} becomes equivalent to a
congruence of the form
$$
x^{d_1}\equiv \lambda_1\pmod p;\quad x\in\N,\quad L+1\le x\le L+N,
$$
for some $\lambda_1\not\equiv 0(\bmod p)$, and we have $d_1|
p-1$ and $N_1=\lfloor p/d_1\rfloor \ge N$. Thus,  without loss
of generality we can assume that $d|p-1$. We can also assume
that
\begin{equation}
\label{eqn:<d<}
p^{\frac{1}{3}-0.001}<d<p^{\frac{2}{3}+0.001},
\end{equation}
as otherwise the claim would follow from the trivial bound
$T_p(d,\lambda,L,N)\le \min\{d, N\}$.

Let $\cG_d$ be the subgroup of $\F_p^*$ of order $d$. We fix one
solution $x=x_0$ to~\eqref{eqn:x^d=lambda}. Clearly,
$T_p(d,\lambda,L,N)$ is equal to the number of solutions of the
congruence
$$
x(\bmod p)\in x_0\cG_d, \quad x\in\N,\quad    L+1\le x\le L+N.
$$

In view of~\eqref{eqn:<d<}, we have
$$
d\in [p^{\frac{1}{3}-0.001}, p^{\frac{1}{3}+0.001}]\cup [p^{\frac{2}{3}-0.001}, p^{\frac{2}{3}+0.001}]\cup [p^{\frac{1}{3}+0.001}, p^{\frac{2}{3}-0.001}].
$$
Accordingly, we consider three cases.

{\bf Case 1.} $p^{\frac{1}{3}-0.001}<d<p^{\frac{1}{3}+0.001}.$
In this case we express $T_p(d,\lambda,L,N)$ in terms of
exponential sums and obtain
$$
T_p(d,\lambda,L,N)=\frac{1}{p}\sum_{a=0}^{p}\sum_{u\in x_0\cG_d}\,\sum_{L+1\le x\le L+N}e_p(a(u-x)).
$$
We separate the term corresponding to $a=0$ and using the
standard arguments and Lemma~\ref{lem:Kon1}, we obtain
\begin{equation}
\label{eqn:Tpdlambda+Kon}
T_p(d,\lambda,L,N)\ll \frac{dN}{p}+d^{29/36}p^{1/18}\Bigl(\frac{1}{p}\sum_{a=1}^{p-1}\Bigl|\sum_{L+1\le x\le L+N}e_p(ax)\Bigr|\Bigr).
\end{equation}
We recall the well-known elementary bound
$$
\frac{1}{p}\sum_{a=1}^{p-1}\Bigl|\sum_{L+1\le x\le L+N}e_p(ax)\Bigr|\ll\log p,
$$
see, for example, the solution to the exercise 11 of Chapter 3
in the book of Vinogradov~\cite{Vin1}. Substituting this
in~\eqref{eqn:Tpdlambda+Kon}, we obtain that
$$
T_p(d,\lambda,L,N)\lesssim d^{29/36}p^{1/18}.
$$
Since $d<p^{\frac{1}{3}+0.001}$, we get
$d^{29/36}p^{1/18}<p^{\frac{1}{3}-0.001}$ and the result follows
in this case.

{\bf Case 2.} $p^{\frac{2}{3}-0.001}<d<p^{\frac{2}{3}+0.001}.$
In this case we denote by $\cT$ the set of integers $x\in
[L+1,L+N]$ for which $x(\bmod p)\in x_0\cG_d$. Then
\begin{equation}
\label{eqn:T=cT}
T_p(d,\lambda,L,N)=|\cT|.
\end{equation}
Clearly, if $x_1,\ldots,x_m\in\cT$, we get that $x_1\cdots
x_m(\bmod p)\in x_0^m\cG_d$. Thus, $|\cT|^m$ is not greater,
than the number of solutions of the congruence
$$
x_1\cdots x_m(\bmod p)\in x_0^m\cG_d,\quad L+1\le x_i\le L+N.
$$
Therefore, for some fixed $\lambda_0\in x_0^m\cG_d$ we have
\begin{equation}
\label{eqn:T m < dR Case 2}
|\cT|^m<d R,
\end{equation}
where $R$ is the number of solutions of the congruence
$$
x_1\cdots x_m\equiv \lambda_0(\bmod p),\quad L+1\le x_i\le L+N.
$$
We express $R$ in terms of character sums and obtain that
$$
R=\frac{1}{p-1}\sum_{\chi}\Bigl(\sum_{L+1\le x\le L+N}\chi(x)\Bigr)^{m}\chi(\lambda_0^{-1}),
$$
where $\chi$ runs through the set of characters modulo $p.$
Separating the term that corresponds to the principal character
$\chi_0$, we get that
$$
R\le \frac{N^m}{p-1}+\max_{\chi\not=\chi_o}\Bigl|\sum_{L+1\le x\le L+N}\chi(x)\Bigr|^m.
$$
We apply Lemma~\ref{lem:Burgess} with $r=5$. Since
$N>0.5p^{\frac{1}{3}-0.001}$, it follows that
$$
\max_{\chi\not=\chi_o}\Bigl|\sum_{L+1\le x\le L+N}\chi(x)\Bigr|\ll N^{59/60}.
$$
Hence, taking $m=200$, we get
$$
R\ll \frac{N^{200}}{p}+\frac{N^{200}}{N^{10/3}}\ll\frac{N^{200}}{p}.
$$
Therefore, from~\eqref{eqn:T m < dR Case 2} we obtain that
$$
|\cT|\ll (d R)^{1/200}\le N\Bigl(\frac{d}{p}\Bigr)^{1/200}\le \Bigl(\frac{p}{d}\Bigr)^{199/200}<p^{\frac{1}{3}-0.0001}.
$$
Hence, substituting this in~\eqref{eqn:T=cT}, we get the desired
estimate.

{\bf Case 3.} $p^{\frac{1}{3}+0.001}<d<p^{\frac{2}{3}-0.001}$.
In particular, we get
$$
N=\lfloor p/d\rfloor\gg p^{\frac{1}{3}+0.001}.
$$

We apply Lemma~\ref{lem:Mont} with
$$
\{\gamma_n\}_n=\Bigl\{\frac{x_0h}{p};\, h\in \cG_d\Bigr\},\quad \alpha=\frac{L+1}{p},\quad \beta=\frac{L+N}{p},\quad K=d.
$$
It follows that
\begin{equation}
\label{eqn:T(d,l,L,N) via Mont}
T_p(d,\lambda,L,N)\ll 1+\frac{1}{d}\sum_{k=1}^d \left|\sum_{h\in \cG_d} e_p(k x_0h)\right|.
\end{equation}
Since $\cG_d$ is cyclic (and therefore consists of all powers of
some element) and $d>p^{\frac{1}{3}+0.001}$, there exists a
subset  $\cU\subset \cG_d$ such that
$$
0.1p^{\frac{1}{3}+0.001}<|\cU|<0.2p^{\frac{1}{3}+0.001},\quad |\cU\cdot \cU|\le 2|\cU|.
$$
Clearly, $r\cG_d=\cG_d$ for any $r\in \cU$. It then follows that
\begin{equation}
\label{eqn:MontApply}
\frac{1}{d}\sum_{k=1}^d \left|\sum_{h\in \cG_d} e_p(k x_0h)\right|=\frac{1}{d|\cU|}\sum_{k=1}^d \sum_{r\in \cU}\left|\sum_{h\in \cG_d} e_p(kr x_0h)\right|.
\end{equation}
Let $I(\mu)$ be the number of solutions of the congruence
$$
kr\equiv \mu\pmod p,\quad k\in \N, \quad k\le d, \quad r\in \cU.
$$
Note that, by Corollary~\ref{cor:Main1}, we have
\begin{equation}
\label{eqn:ku=k1u1}
\sum_{\mu=0}^{p-1}I(\mu)^2\lesssim \frac{|\cU|^{7/4}d}{p^{1/4}}.
\end{equation} Indeed, the left hand side is equal to the number of solutions of the congruence
$$
kr\equiv k_1r_1\pmod p,\quad k,k_1\in \N; \quad k,k_1\le d; \quad r,r_1\in \cU.
$$
Moreover,
$$
p^{1/3}<|\cU|<p^{2/5},\quad |\cU\cdot \cU|<10|\cU|,\quad |\cU|d<p^{\frac{2}{3}-0.001}p^{\frac{1}{3}+0.001}=p.
$$
Thus, we are at the condition of $(ii)$ of
Corollary~\ref{cor:Main1}. Therefore, the
estimate~\eqref{eqn:ku=k1u1} holds.

Now, we use the Cauchy-Schwarz inequality and obtain
\begin{eqnarray*}
\begin{split}\sum_{k=1}^d \sum_{r\in \cU}\left|\sum_{h\in \cG_d} e_p(kr x_0h)\right|=\sum_{\mu=0}^{p-1}I(\mu)\left|\sum_{h\in \cG_d} e_p(\mu x_0h)\right|\\ \le
\Bigl(\sum_{\mu=0}^{p-1}I(\mu)^2\Bigr)^{1/2} \Bigl(\sum_{\mu=0}^{p-1}\left|\sum_{h\in \cG_d} e_p(\mu x_0h)\right|^2\Bigr)^{1/2}\\ \lesssim  \frac{|\cU|^{7/8}d^{1/2}}{p^{1/8}}(pd)^{1/2}
=|\cU|^{7/8}dp^{3/8}.
\end{split}
\end{eqnarray*}
Substituting this in~\eqref{eqn:MontApply}, we obtain
$$
\frac{1}{d}\sum_{k=1}^d \left|\sum_{h\in \cG_d} e_p(k x_0h)\right|\lesssim \frac{p^{3/8}}{|\cU|^{1/8}}\lesssim p^{\frac{1}{3}-0.0001}.
$$
This together with~\eqref{eqn:T(d,l,L,N) via Mont} proves
Theorem~\ref{thm:x^n=lambda}.

We shall now derive Corollary~\ref{cor:x f(x)}. Let $J$ be the
number of solutions of~\eqref{eqn:Th2}. Clearly, if
$\gcd(f(x),p-1)=d$, then~\eqref{eqn:Th2} implies that
$x^{d}\equiv 1 \pmod p.$ Hence,
\begin{equation}
\label{eqn:J= sum d | p-1 Jd}
J=\sum_{d|p-1}J_d,
\end{equation}
where $J_d$ is the number of solutions of the congruence
$$
x^{d}\equiv 1\pmod p; \quad x\in \N, \quad x<p, \quad \gcd(f(x),p-1)=d.
$$
Since $f(x)$ does not have multiple roots, by the Nagell-Ore
theorem (see~\cite{Hux}, even for a stronger form) the set of
$x$ with $f(x)\equiv 0\pmod d$ consists on the union of
arithmetic progressions of the form $x\equiv k\pmod d$ for at
most $d^{o(1)}$ different values of $k$. Thus, for each $d|p-1$
there exists a non-negative integer $k_0<d$ such that
\begin{equation}
\label{eqn:J_d lesssim J_d'}
J_d\lesssim J_d',
\end{equation}
where $J_d'$ is the number of solutions of the congruence
$$
x^{d}\equiv 1\pmod p; \quad x\in \N, \quad x<p, \quad x\equiv k_0\pmod d.
$$
Representing $x=k_0+dy$, we get the congruence
$$
(k_0+dy)^{d}\equiv 1\pmod p; \quad y\in \N\cup\{0\}, \quad y<p/d.
$$
Hence, if we denote by $d_1$ the multiplicative inverse of
$d\pmod p$, we get that $J_d'$ is not greater than the number of
solutions of the congruence
$$
(y+k_0d_1)^{d}\equiv d_1^d\pmod p;\quad y\in\N\cup\{0\},\quad y<p/d.
$$
According to Theorem~\ref{thm:x^n=lambda}, we have $J_d'\lesssim
p^{\frac{1}{3}-c}$ for some absolute constant $c>0$. Combining
this bound with~\eqref{eqn:J_d lesssim J_d'} and~\eqref{eqn:J=
sum d | p-1 Jd}, we conclude the proof.

\section{Proof of Theorem~\ref{thm:Main1}}

We first establish the following statement, based on
Corollary~\ref{cor:Main1}.

\begin{lemma}
\label{lem:Main2} Let $0<\varepsilon<0.01$ be fixed, $\cU\subset
\F_p^*$ be such that $\quad |\cU\cdot\cU|\le 10|\cU|$ and
$$
H=\lfloor p^{1/4+\varepsilon}\rfloor,\quad 2H<|\cU|\le p^{3/8-0.5\varepsilon}.
$$ Then the number $T$ of solutions of the congruence
\begin{equation}
\label{eqn:qxh=qxh}
qxr\equiv q_1x_1r_1\pmod p
\end{equation}
in positive integers $x,x_1$, prime numbers $q,q_1$ and elements
$r,r_1\in\cU$ with
\begin{equation}
\label{eqn:conditions for qxh=qxh}
x,x_1\le H,\quad 0.5|\cU|< q,q_1\le |\cU|
\end{equation}
satisfies
$$
T\lesssim |\cU|^2H.
$$
\end{lemma}

\begin{proof}

We have
$$
T=T_1+T_2,
$$
where $T_1$ is the number of solutions of~\eqref{eqn:qxh=qxh}
satisfying~\eqref{eqn:conditions for qxh=qxh} with the
additional condition $q=q_1$, and $T_2$ is the number of
solutions with $q\not=q_1$. We observe that
$|\cU|<p^{3/8}<p^{2/5}$ and
$$|\cU|H^2<p^{3/8-0.5\epsilon}p^{1/2+2\epsilon}<p^{7/8+3\epsilon/2}<p.$$
Thus, we can apply Corollary~\ref{cor:Main1} with $n=2$ and get
$$
T_1\lesssim |\cU|^2H.
$$
In order to estimate $T_2$,  we fix  $x_1,r,r_1$ such that
$$
T_2\le |\cU|^2H T'_2,
$$
where $T'_2$ is the number of solutions of the congruence
$$
\frac{qx}{q_1}\equiv \frac{x_1r_1}{r}\pmod p
$$
in positive integers $x\le H$ and prime numbers $q,q_1$ with
$$
q\not=q_1, \quad 0.5|\cU|<q,q_1\le |\cU|.
$$
From $x\le H< q_1$, it follows that $\gcd(qx,q_1)=1$. Since
$H|\cU|^2<p$, from Lemma~\ref{lem:simple u/v=a} we get that $xq$
and $q_1$ are uniquely determined. Since $x<q$, the value $xq$
uniquely determines $x$ and $q$. Hence, $T'_2\le 1$, whence
$T_2\le |\cU|^2H$ concluding the proof of our lemma.

\end{proof}

We now proceed to prove Theorem~\ref{thm:Main1}. Assuming
$\varepsilon<0.01$, we define
$$
m_0=\Bigl\lceil\frac{1}{\varepsilon}\Bigr\rceil,\quad L=\lfloor p^{1/m_0}\rfloor,\quad H=\lfloor p^{1/4+\varepsilon}\rfloor,\quad N=\lfloor p^{3/8-0.5\varepsilon}\rfloor.
$$
From Lemma~\ref{lem:Burgess} it follows that there exists
$\delta=\delta(\varepsilon)>0$ such that for any non-principal
character $\chi$ modulo $p$ the following bound holds:
\begin{equation}
\label{eqn:Burg}
\left|\sum_{x\le H}\chi(x)\right|\le H^{1-\delta}.
\end{equation}
Let $\cG'$ be a subset of $\cG$ such that $|\cG'|=N$ and
$|\cG'\cG'|\le 2|\cG'|$. The existence of such a subset is
obvious, since either $\cG$ consists on consecutive powers of a
primitive root or it is a subgroup of $\F_p^*$, which is cyclic.

It suffices to prove that for some
$\delta_0=\delta_0(\varepsilon)>0$ there are
$p+O(p^{1-\delta_0})$ residue classes modulo $p$ of the form
$zxqr\pmod p,$ with positive integers $x,z$, prime numbers $q$
and elements $r$ satisfying
$$
z\le L,\quad x\le H,\quad \frac{N}{2}<q\le N,\quad r\in \cG'.
$$
Let $\Lambda\subset\F_p^*$ be the exceptional set, that is,
assume that the congruence
$$
zxqr\equiv \lambda\pmod p
$$
has no solutions in $\lambda\in \Lambda$ and $z,x,q,r$ as above.
We write this condition in the form of character sums
$$
\frac{1}{p-1}\sum_{\chi}\sum_{z\le L}\sum_{x\le H}\,\sum_{\substack{0.5N<q<N\\ q \,\, {\text {is prime}}}}\sum_{r\in\cG'}
\sum_{\lambda\in\Lambda}\chi(zxqr\lambda^{-1})=0,
$$
where $\chi$ runs through the set of characters modulo $p$.
Separating the term corresponding to the principal character
$\chi=\chi_0$, we get
$$
LHN^2|\Lambda|\lesssim \sum_{\chi\not=\chi_0}\Bigl|\sum_{z\le L}\chi(z)\Bigr| \Bigl|\sum_{x\le H}\,\sum_{\substack{0.5N<q<N\\ q\,\, {\text {is prime}}}}\sum_{r\in\cG'}
\chi(xqr)\Bigr| \Bigl|\sum_{\lambda\in\Lambda}\chi(\lambda)\Bigr|.
$$

Following~\cite{JG, Gar1}, we split the set of nonprincipal
characters $\chi$ into two subsets $\cX_1$ and $\cX_2$ as
follows. To the set $\cX_1$  we allot those characters $\chi$,
for which
$$
\Bigl|\sum_{z\le L}\chi(z)\Bigr|\ge L^{1-0.1\delta},
$$
where $\delta$ is defined from~\eqref{eqn:Burg}. The remaining
characters we include to the set $\cX_2$, these are the
characters that satisfy
$$
\Bigl|\sum_{z\le L}\chi(z)\Bigr| < L^{1-0.1\delta}.
$$
Thus, we have
\begin{equation}
\label{eqn:le W1+W2}
LHN^2|\Lambda|\lesssim W_1+W_2,
\end{equation}
where
$$
W_i=\sum_{\chi\in\cX_i}\Bigl|\sum_{z\le L}\chi(z)\Bigr| \Bigl|\sum_{x\le H}\,\sum_{\substack{0.5N<q<N\\ q\,\, {\text {is prime}}}}\sum_{r\in\cG'}
\chi(xqr)\Bigr| \Bigl|\sum_{\lambda\in\Lambda}\chi(\lambda)\Bigr|.
$$

To deal with $W_1$, we show that the cardinality of $\cX_1$ is
small. We have
$$
|\cX_1|L^{2m_0(1-0.1\delta)}\le \sum_{\chi\in\cX_1}\Bigl|\sum_{z\le L}\chi(z)\Bigr|^{2m_0}\le \sum_{\chi}\Bigl|\sum_{z\le L}\chi(z)\Bigr|^{2m_0}=(p-1)T,
$$
where $T$ is the number of solutions of the congruence
$$
x_1\cdots x_{m_0}\equiv y_1\cdots y_{m_0}\pmod p;\quad x_i,y_j\in \mathbb{N};\quad x_i,y_j\le L.
$$
Since $L^{m_0}<p$, the congruence is converted to an equality
and we have, by the estimate for the divisor function, at most
$L^{m_0+o(1)}$ solutions. Thus,
$$
|\cX_1|L^{2m_0(1-0.1\delta)}\lesssim pL^{m_0+o(1)},
$$
whence, in view of $L^{m_0}=p^{1+o(1)}$, we get
$$
|\cX_1|\lesssim p^{0.2\delta}.
$$
Thus, estimating in $W_1$ the sums over $z,q, r,\lambda$
trivially and applying~\eqref{eqn:Burg} to the sum over $x$, we
get
\begin{equation*}
\begin{split}
W_1=\sum_{\chi\in\cX_1}\Bigl|\sum_{z\le L}\chi(z)\Bigr| \Bigl|\sum_{x\le H}\chi(x)\Bigr|\Bigl|\sum_{\substack{0.5N<q<N\\ q\,\, {\text {is prime}}}}\sum_{r\in\cG'}
\chi(qr)\Bigr| \Bigl|\sum_{\lambda\in\Lambda}\chi(\lambda)\Bigr|\\ \lesssim |\cX_1|\cdot L\cdot N^2\cdot |\Lambda|\max_{\chi\not=\chi_0}\Bigl|\sum_{x\le H}\chi(x)\Bigr|\lesssim p^{0.2\delta}LN^2|\Lambda|H^{1-\delta}.
\end{split}
\end{equation*}
Therefore, since $H>p^{1/4}$ we have, for sufficiently large
$p$, the estimate
$$
W_1< LHN^2|\Lambda|p^{-0.01\delta}.
$$
Inserting this bound into \eqref{eqn:le W1+W2}, we get
\begin{equation}
\label{eqn:LHN2Lambda le W2}
LHN^2|\Lambda|\lesssim W_2.
\end{equation}
We next estimate $W_2$. By the definition of $\cX_2$, we have
\begin{equation}
\label{eqn:W2 le}
W_2\le L^{1-0.1\delta}\sum_{\chi}\Bigl|\sum_{x\le H}\sum_{\substack{0.5N<q<N\\ q\,\, {\text {is prime}}}}\sum_{r\in\cG'}
\chi(xqr)\Bigr| \Bigl|\sum_{\lambda\in\Lambda}\chi(\lambda)\Bigr|.
\end{equation}
Next, we have
$$
\sum_{\chi}\Bigl|\sum_{\lambda\in\Lambda}\chi(\lambda)\Bigr|^2=(p-1)|\Lambda|
$$
and
$$
\sum_{\chi}\Bigl|\sum_{x\le H}\sum_{\substack{0.5N<q<N\\ q\,\, {\text {is prime}}}}\sum_{r\in\cG'}
\chi(xqr)\Bigr|^2=(p-1)T,
$$
where $T$ is the number of solutions of the congruence
$$
xqr\equiv x_1q_1r_1\pmod p,
$$
in positive integers $x,x_1$ prime numbers $q,q_1$ and elements
$r,r_1\in\cG'$ satisfying
$$
x_1, x_2\le H,\quad 0.5N<q, q_1<N,\quad r,r_1\in\cG'.
$$
From Lemma~\ref{lem:Main2} with $\cU=\cG'$ it follows that
$$
T\lesssim N^2H.
$$
Therefore, applying the Cauchy-Schwarz inequality
in~\eqref{eqn:W2 le}, and using~\eqref{eqn:LHN2Lambda le W2}, we
obtain that
$$
L^2H^2N^4|\Lambda|^2\lesssim W_2^2\lesssim L^{2-0.02\delta}(p-1)^2|\Lambda| T\lesssim L^{2-0.02\delta}p^2 |\Lambda| N^2 H.
$$
Thus,
$$
|\Lambda|\lesssim \frac{p^2L^{-0.02\delta}}{HN^2}\lesssim p L^{-0.02\delta}.
$$
Therefore,
$$
|\Lambda|< p^{1-\delta_0}
$$
for some $\delta_0=\delta_0(\epsilon)$, which concludes the
proof.

\section{Proof of Theorem~\ref{thm:Main2}}

The proof follows the same line as the proof of
Theorem~\ref{thm:Main1}, however, instead of
Lemma~\ref{lem:Main2} we shall use Lemma~\ref{lem:Main3} .

\begin{lemma}
\label{lem:Main3} Let $0<\varepsilon<0.01$ be fixed, $\cU\subset
\F_p^*$ be such that $\quad |\cU\cdot\cU|\le 10|\cU|$ and
$$
H=\lfloor p^{1/4+\varepsilon}\rfloor,\quad Q=\lfloor p^{1/4} \rfloor,\quad |\cU|\le p^{1/4-\varepsilon}.
$$
Then the number $T$ of solutions of the congruence
\begin{equation}
\label{eqn:qqxh=qqxh}
q_1q_2xr\equiv q_1' q_2'x'r'\pmod p
\end{equation}
in positive integers $x,x'$, prime numbers $q_1, q_2, q_1',
q_2'$ and elements $r,r'\in\cU$ with
\begin{equation}
\label{eqn:condition for qqxh=qqxh}
x,x'\le H,\quad \frac{Q}{4}< q_1,q'_1< \frac{Q}{2},\quad \frac{Q}{2}< q_2,q'_2< Q,
\end{equation}
satisfies
$$
T\lesssim Q^2H|\cU|.
$$
\end{lemma}

Let us prove the lemma. We have
\begin{equation}
\label{eqn:T=T1+T2+T3+T4}
T=T_1+T_2+T_3+T_4,
\end{equation}
where $T_1$ is the number of solutions of~\eqref{eqn:qqxh=qqxh}
satisfying~\eqref{eqn:condition for qqxh=qqxh} and with the
additional condition $q_1=q_2,\, q_1'=q_2',$  $T_2$ is the
number of solutions with $q_1\not=q_1',\, q_2\not=q_2'$,  $T_3$
is the number of solutions with $q_1=q_1',\, q_2\not=q_2'$ and
$T_4$ is the number of solutions with $q_1\not=q_1',\,
q_2=q_2'$.

We have $T_1\le Q^2 T_1'$, where $T_1'$ is the number of
solutions of the congruence
$$
xr\equiv x'r'\pmod p;\quad x,x'\le H,\quad r,r'\in\cU.
$$
Applying Corollary~\ref{cor:Main1} with $n=1$ or $n=2$, we get
that $T_1'\lesssim |\cU|H$. Therefore,
\begin{equation}
\label{eqn:T1 lesssim}
T_1\lesssim Q^2H|\cU|.
\end{equation}

To estimate $T_2$, we fix $x,r,x',r'$ and see that
$$
T_2\lesssim |\cU|^2|H|^2 T_2',
$$
where $T_2'$ is the number of solutions of the congruence
$$
\frac{q_1q_2}{q_1'q_2'}\equiv \frac{x'r'}{xr}\pmod p
$$
in prime numbers $q_1,q_1',q_2,q_2'$ with
$$
\frac{Q}{4}< q_1,q'_1< \frac{Q}{2},\quad \frac{Q}{2}< q_2,q'_2< Q,\quad q_1\not=q_1',\quad q_2\not=q_2'.
$$
Since $\gcd(q_1q_2,q_1'q_2')=1$ and $Q^4<p$, it follows from
Lemma~\ref{lem:simple u/v=a} that the numbers $q_1q_2$ and
$q_1'q_2'$ are uniquely determined. Since $q_1<q_2$ and
$q_1'<q_2'$, this implies that in fact all the prime numbers
$q_1,q_2,q_1',q_2'$ are uniquely determined. Therefore, $T_2'\le
1$ and we get
\begin{equation}
\label{eqn:T2 lesssim}
T_2\lesssim |\cU|^2|H|^2\lesssim Q^2H|\cU|.
\end{equation}

In order to estimate $T_3$ and $T_4$, we note that
$$
T_3+T_4\le QT_5,
$$
where $T_5$ is the number of solutions of the congruence
$$
qxr\equiv q'x'r'\pmod p,
$$
in positive integers $x,x'$, prime numbers $q,q'$ and elements
$r,r'\in \cU$ with
$$
x,x'\le H,\quad \frac{Q}{4}<q,q'<Q,\quad q\not=q'.
$$
We introduce variables $x_1,x_2$ with
$$
x_1=qx,\quad x_2=q'x'
$$
and note that
$$
x_1r\equiv x_2r'\pmod p;\quad x_1, x_2\le QH, \quad  r,r'\in\cU.
$$
We can apply Corollary~\ref{cor:Main1} with $n=1$ and $H$
substituted by $QH$ (clearly, the conditions of
Corollary~\ref{cor:Main1} are satisfied). It then follows that
there are at most $QH|\cU|p^{o(1)}$ possibilities for the
quadruple $(x_1,x_2,r,r')$. Each such quadruple determines
$q,x,q',x'$ with at most $p^{o(1)}$ possibilities, because
$q,x,q',x'$ are divisors of $x_1x_2<p.$  Therefore, we get that
$$
T_5\lesssim QH|\cU|,
$$
implying that
$$
T_3+T_4\lesssim Q^2H|\cU|.
$$
Inserting this estimate together with~\eqref{eqn:T1 lesssim}
and~\eqref{eqn:T2 lesssim} into~\eqref{eqn:T=T1+T2+T3+T4}, we
conclude the proof of our lemma.

Now we proceed to prove Theorem~\ref{thm:Main2}. Assuming
$\varepsilon<0.01$, we define
$$
m_0=\Bigl\lceil\frac{1}{\varepsilon}\Bigr\rceil,\quad L=\lfloor p^{1/m_0}\rfloor,\quad H=\lfloor p^{1/4+\varepsilon}\rfloor,\quad Q=\lfloor p^{1/4}\rfloor,\quad  N=\lfloor p^{1/4-\varepsilon}\rfloor.
$$
Following the proof of Theorem~\ref{thm:Main1}, we denote by
$\delta=\delta(\varepsilon)>0$ a positive constant such that for
any non-principal character $\chi$ modulo $p$ the following
bound holds:
\begin{equation}
\label{eqn:BurgT3}
\left|\sum_{x\le H}\chi(x)\right|\le H^{1-\delta}.
\end{equation}
We again let $\cG'$ be a subset of $\cG$ such that $|\cG'|=N$
and $|\cG'\cG'|\le 2|\cG'|$.

It suffices to prove that for some
$\delta_0=\delta_0(\varepsilon)>0$ there are
$p+O(p^{1-\delta_0})$ residue classes modulo $p$ of the form
$zxq_1q_2r\pmod p,$ with positive integers $z,x$, prime numbers
$q_1,q_2$ and elements $r$ satisfying
$$
z\le L,\quad x\le H,\quad \frac{Q}{4}<q_1\le \frac{Q}{2},\quad \frac{Q}{2}<q_2\le Q,\quad r\in \cG'.
$$
Let $\Lambda\subset\F_p^*$ be the exceptional set, that is,
assume that the congruence
$$
zxq_1q_2r\equiv \lambda\pmod p
$$
has no solutions in $\lambda\in \Lambda$ and $z,x,q_1,q_2,r$ as
above. Following the proof of Theorem~\ref{thm:Main1}, we derive
that
$$
LHQ^2N|\Lambda|\lesssim \sum_{\chi\not=\chi_0}\Bigl|\sum_{z\le L}\chi(z)\Bigr| \Bigl|\sum_{x\le H}\,\sum_{\substack{Q/4<q_1<Q/2\\ Q/2<q_2<Q\\ q_1, q_2\,\, {\text
{are primes}}}}\sum_{r\in\cG'}
\chi(xq_1q_2r)\Bigr| \Bigl|\sum_{\lambda\in\Lambda}\chi(\lambda)\Bigr|.
$$

We define the set of characters  $\cX_1$ and $\cX_2$ exactly the
same way as in the proof of Theorem~\ref{thm:Main2} and write
\begin{equation}
\label{eqn:le W1+W2 T3}
LHQ^2N|\Lambda|\lesssim W_1+W_2,
\end{equation}
where
$$
W_i=\sum_{\chi\in\cX_i}\Bigl|\sum_{z\le L}\chi(z)\Bigr| \Bigl|\sum_{x\le H}\,\sum_{\substack{Q/4<q_1<Q/2\\ Q/2<q_2<Q\\ q_1, q_2\,\, {\text {are primes}}}}\sum_{r\in\cG'}
\chi(xq_1q_2r)\Bigr| \Bigl|\sum_{\lambda\in\Lambda}\chi(\lambda)\Bigr|.
$$
From the proof of Theorem~\ref{thm:Main1} it follows that
$$
|\cX_1|\lesssim p^{0.2\delta}.
$$
Thus, estimating in $W_1$ the sums over $z, q_1, q_2,r,\lambda$
trivially and applying~\eqref{eqn:BurgT3} to the sum over $x$,
we get
$$
W_1< LHQ^2N|\Lambda|p^{-0.01\delta}.
$$
Inserting this bound into \eqref{eqn:le W1+W2 T3}, we get
$$
LHQ^2N|\Lambda|\lesssim W_2.
$$
Following the argument of the proof of Theorem~\ref{thm:Main1}
we have
$$
L^2H^2Q^4N^2|\Lambda|^2\lesssim W_2^2\lesssim L^{2-0.02\delta}(p-1)^2|\Lambda| T,
$$
where $T$ is the number of solutions of the congruence
$$
xq_1q_2r\equiv x'q_1'q_2'r'\pmod p,
$$
in positive integers $x,x'$ prime numbers $q_1,q_2,q_1',q_2'$
and elements $r,r_1\in\cG'$ satisfying
$$
x_1, x_2\le H,\quad \frac{Q}{4}<q_1, q_1'<\frac{Q}{2},\quad \frac{Q}{2}<q_1, q_1'<Q, \quad r,r_1\in\cG'.
$$
From Lemma~\ref{lem:Main3} with $\cU=\cG'$ it follows that
$$
T\lesssim HQ^2N.
$$
Therefore,
$$
L^2H^2Q^4N^2|\Lambda|^2\lesssim L^{2-0.02\delta}p^2 H Q^2 |\Lambda|.
$$
Thus,
$$
|\Lambda|\lesssim\frac{p^2}{HQ^2N} L^{-0.02\delta},
$$
whence
$$
|\Lambda|<p^{1-\delta_0}
$$
for some $\delta_0=\delta_0(\epsilon)$. This finishes the proof
of Theorem~\ref{thm:Main2}.

Address of the authors:\\

J. Cilleruelo, Instituto de Ciencias Matem\'{a}ticas
(CSIC-UAM-UC3M-UCM) and Departamento de Matem\'{a}ticas,
Universidad Aut\'{o}noma de Madrid, Madrid-28049, Spain.

Email:{\tt franciscojavier.cilleruelo@uam.es}.

\vspace{1cm}

M.~Z.~Garaev, Centro de Ciencias Matem\'{a}ticas,  Universidad
Nacional Aut\'onoma de M\'{e}xico, C.P. 58089, Morelia,
Michoac\'{a}n, M\'{e}xico,

Email:{\tt garaev@matmor.unam.mx}

\end{document}